\documentclass[12pt]{amsart}
\usepackage{fullpage}
\usepackage{amsmath}
\usepackage{amsthm}
\usepackage{amssymb}
\usepackage{stmaryrd}

\newtheorem{thm}{Theorem}[section]

\theoremstyle{definition}

\newtheorem*{ex}{Example}
\newtheorem*{rmk}{Remark}

\newcommand{\lbrac}{\llbracket}
\newcommand{\rbrac}{\rrbracket}
\newcommand{\CC}{\mathbb C}
\newcommand{\ZZ}{\mathbb Z}
\newcommand{\RR}{\mathbb R}
\newcommand{\QQ}{\mathbb Q}
\newcommand{\Vol}{\operatorname{Vol}}
\newcommand{\brac}[2]{\genfrac{\langle}{\rangle}{0pt}{}{#1}{#2}}

\title{Laurent polynomials, Eulerian numbers, and Bernstein's theorem}
\author{Ricky Ini Liu}
\address{Department of Mathematics, University of Michigan, 
Ann Arbor, MI 48109}
\email{riliu@umich.edu}
\begin{document}
\begin{abstract}
Erman, Smith, and V\'arilly-Alvarado \cite{ESV} showed that the expected number of doubly monic Laurent polynomials $f(z) = z^{-m} + a_{-m+1}z^{-m+1} + \cdots + a_{n-1}z^{n-1} + z^n$ whose first $m+n-1$ powers have vanishing constant term is the Eulerian number $\brac{m+n-1}{m-1}$, as well as a more refined result about sparse Laurent polynomials. We give an alternate proof of these results using Bernstein's theorem that clarifies the connection between these objects. In the process, we show that a refinement of Eulerian numbers gives a combinatorial interpretation for volumes of certain rational hyperplane sections of the hypercube.
\end{abstract}

\maketitle

\section{Introduction}

Fix positive integers $m$ and $n$, and let $f(z) = z^{-m} + a_{-m+1}z^{-m+1} + \cdots + a_{n-1}z^{n-1} + z^n$ be a doubly monic Laurent polynomial (with complex coefficients). Denote by $\lbrac f^k \rbrac$, for all positive integers $k$, the constant coefficient of the $k$th power of $f(z)$. Then, motivated by a result of Duistermaat and van der Kallen \cite{DK} and a related question by Sturmfels \cite{Sturmfels}, Erman, Smith, and V\'arilly-Alvarado \cite{ESV} proved the following theorem. 

\begin{thm}[\cite{ESV}]\label{thm-main}
The expected number of Laurent polynomials $f(z)$ such that $\lbrac f^1\rbrac = \lbrac f^2\rbrac = \cdots = \lbrac f^{m+n-1}\rbrac=0$ is the Eulerian number $\brac{m+n-1}{m-1}$.
\end{thm}
Recall that the Eulerian number $\brac{n}{k}$ is the number of permutations in $S_n$ with exactly $k$ descents. (In \cite{ESV}, the ``expected number'' is computed as the degree of the ideal \[I_{m,n} = \langle \lbrac f^1\rbrac, \lbrac f^2\rbrac, \cdots, \lbrac f^{m+n-1}\rbrac\rangle\] in $\CC[a_{-m+1}, \cdots, a_{n-1}]$, assuming this ideal is zero-dimensional.) 

Their proof of this theorem involves computing the Chow ring of an appropriate toric variety and then showing that the desired answer satisfies a certain recurrence that is also satisfied by the Eulerian numbers. As a result, the appearance of the Eulerian numbers using this approach is somewhat mysterious.

In this paper, we will attempt to elucidate this connection by giving an alternate proof of this result. We will use Bernstein's Theorem (also known as the BKK Theorem after Bernstein \cite{Bernstein}, Kushnirenko \cite{Kushnirenko}, and Khovanskii \cite{Khovanskii}) to relate the number of solutions of an appropriate polynomial system to the normalized volume of a hypersimplex, which is known to be enumerated by an Eulerian number in a natural way (see \cite{Stanley}). We will also, again following \cite{ESV}, apply a similar technique to show the analogous result for ``sparse'' Laurent polynomials. In the process, we will show that the volume of the intersection of the unit hypercube $[0,1]^n$ with a hyperplane of the form $\sum_{i=1}^n x_i = \frac cd$ for $\frac cd \in \QQ$ can be interpreted combinatorially in terms of a refinement of the Eulerian numbers. (This refinement is a special case of the one given by the cyclic sieving phenomenon described in \cite{SSW}.)

\section{Preliminaries}
In this section, we review the two main ingredients needed for our proof, namely Bernstein's theorem on counting solutions to generic polynomial systems, and the relationship between Eulerian numbers and hypersimplices.

\subsection{Bernstein's theorem}

Given a Laurent polynomial $P(x_1, \dots, x_n) = \sum c_a x^a$ (where $x^a = x_1^{a_1}x_2^{a_2} \cdots x_n^{a_n}$), its \emph{support} is the set of vectors $a = (a_1, \dots, a_n) \in \ZZ^n$ such that $c_a$ is nonzero. The \emph{Newton polytope} $\Delta_P \subset \RR^n$ of $P$ is the convex hull of the support of $P$. Bernstein's theorem gives a bound on the number of solutions to a polynomial system $P_1 = P_2 = \dots = P_n = 0$ that lie in the algebraic torus $(\CC^\times)^n$ in terms of the Newton polytopes $\Delta_{P_1}, \dots, \Delta_{P_n}$.

More precisely, given polytopes $\Delta_1, \dots, \Delta_n \subset \RR^n$, the volume of the polytope $t_1\Delta_1 + \dots + t_n \Delta_n$ (where $+$ denotes Minkowski sum) is a polynomial in $t_1, \dots, t_n$. Then the \emph{mixed volume} $MV(\Delta_1, \dots, \Delta_n)$ is the coefficient of $t_1t_2\cdots t_n$ in this polynomial. Bernstein's theorem can then be stated as follows.

\begin{thm}[\cite{Bernstein}] \label{thm-bernstein}
Let $P_1, \dots, P_n \in \CC[x_1, \dots, x_n]$ with Newton polytopes $\Delta_1, \dots, \Delta_n$. Then the polynomial system $P_1=P_2=\dots=P_n=0$ has at most $MV(\Delta_1,\dots,\Delta_n)$ isolated solutions in $(\CC^\times)^n$. Moreover, if $P_1, \dots, P_n$ have generic coefficients given $\Delta_1, \dots, \Delta_n$, then this bound is exact.
\end{thm}

Bernstein's theorem can be proved, for instance, using toric geometry or homotopy continuation methods.

\subsection{Hypersimplices}

Let $k$ and $n$ be positive integers with $0<k<n$. The \emph{hypersimplex} $\Delta_{k,n} \subset \RR^n$ is the polytope
\[\Delta_{k,n} = \{(x_1, \dots, x_n) \mid x_1+x_2+\cdots+x_n = k, \;\; x_i \in [0,1]\}.\] 
In other words, $\Delta_{k,n}$ is a slice of the hypercube $[0,1]^n$ by the hyperplane $\sum x_i=k$. Note that $\Delta_{k,n}$ has dimension $n-1$, and $\Delta_{1, n}$ is the standard $(n-1)$-simplex. Normalizing so that the $(n-1)$-dimensional volume of $\Delta_{1, n}$ is 1, we have the following theorem of Stanley.
\begin{thm}[\cite{Stanley}]\label{thm-stanley}
The normalized volume of $\Delta_{k,n}$ is the Eulerian number $\brac{n-1}{k-1}$.
\end{thm}
Here the Eulerian number $\brac{n}{k}$ is the number of permutations in $S_n$ with exactly $k$ descents. The proof given in \cite{Stanley} constructs an explicit bijection between permutations in $S_{n-1}$ with $k-1$ descents and simplices in a  unimodular triangulation of $\Delta_{k,n}$.

\section{Proof of the main theorem}

We are now ready to give a proof of Theorem~\ref{thm-main}.

\begin{proof}[Proof of Theorem~\ref{thm-main}]
Let $N=m+n$, and write 
\[f(z) = z^{-m} \prod_{i=1}^N (1+r_iz),\]
so that $\prod r_i = 1$. Then the map sending $r = (r_1, \dots, r_N)$ to $f(z)$ has degree $N!$, so it suffices to show that the expected number of possible $r$ is $N! \cdot \brac{N-1}{m-1}$.

Since the constant term $\lbrac f^k \rbrac$ is the coefficient of $z^{mk}$ in $\prod_{i=1}^N (1+r_iz)^k$, it can be expressed as a symmetric polynomial $P_k(r_1, \dots, r_N)$. The monomial $r_1^{a_1}\cdots r_N^{a_N}$ appears in $P_k$ with coefficient $\prod_{i=1}^N \binom{k}{a_i}$ provided that $\sum a_i = mk$. Hence the support of $P_k$ is the set of all exponent vectors $a=(a_1, \dots, a_N)$ such that $\sum a_i = mk$ and $0 \leq a_i \leq k$. Thus the Newton polytope $\Delta_{P_k}$ is just $k\Delta_{m, N}$.

By Bernstein's theorem, the expected number of solutions to the system \[P_1(r)=\cdots = P_{N-1}(r)=r_1r_2\cdots r_N-1 = 0\] (noting that we must have $r_i \in \CC^\times$) is \[MV(\Delta_{m,N}, 2\Delta_{m,N}, \cdots, (N-1)\Delta_{m,N}, I),\] where $I$ is the segment between $(0,0,\dots, 0)$ and $(1,1,\dots, 1)$. This is the coefficient of $t_1t_2\cdots t_N$ in the volume of $\Delta(t_1,\dots, t_N) = (t_1+2t_2+\cdots + (N-1)t_{N-1})\Delta_{m,N} + t_NI$.

Let $\Lambda \subset \RR^N$ be the hyperplane $\sum a_i = 0$ (so that $\Delta_{m,N}$ lies in a hyperplane parallel to $\Lambda$ and $I$ is orthogonal to $\Lambda$). Then $\ZZ^{N-1} \cap \Lambda$ and $(1,1,\dots, 1)$ span a lattice that has index $N$ in $\ZZ^N$. Thus, by Theorem~\ref{thm-stanley}, $\Vol (\Delta_{m,N}+I) = \frac{N}{(N-1)!}\brac{N-1}{m-1}$, and hence
\[\Vol \Delta(t_1, \dots, t_N) = (t_1+2t_2+\cdots+(N-1)t_{N-1})^{N-1}t_N \cdot \frac{N}{(N-1)!}\brac{N-1}{m-1}.\]
The coefficient of $t_1\cdots t_N$ is then $N! \cdot \brac{N-1}{m-1}$, as desired.
\end{proof}

\section{Refinement}

The authors of \cite{ESV} also count the expected number of Laurent polynomials as in Theorem~\ref{thm-main} that are ``sparse'' in a certain sense. Let us call $f(z)$ \emph{$d$-sparse} for some $d$ dividing $N=m+n$ if the coefficient of $z^i$ vanishes unless $i \equiv n \pmod d$. (Note that every $f(z)$ is $1$-sparse.) We will adapt our technique from above to count the expected number of $d$-sparse $f(z)$ satisfying Theorem~\ref{thm-main}. 

We first define a fractional analogue of the hypersimplex. Given a positive integer $n$ and a rational number $\frac{c}{d}$ with $0<\frac{c}{d}<n$ and $\gcd(c,d)=1$, define the rational polytope 
\[\Delta_{\frac{c}{d},n} = \{(x_1, \dots, x_n)\mid x_1 + \cdots + x_n = \tfrac{c}{d}, x_i \in [0,1]\}.\]
In other words, $\Delta_{\frac{c}{d},n}$ is a slice of the hypercube $[0,1]^n$ by the hyperplane $\sum x_i = \frac{c}{d}$. Although $\Delta_{\frac{c}{d},n}$ is not an integer polytope, $d\Delta_{\frac{c}{d},n}$ is. We will give a combinatorial interpretation for the normalized volume of $d\Delta_{\frac{c}{d},n}$.

We may consider the permutations in $S_n$ as circular permutations of $\{0, 1, \dots, n\}$ (that is, permutations of $\{0, 1, \dots, n\}$ up to cyclic shift). Then a permutation with $k$ descents corresponds to a circular permutation with $k+1$ (cyclic) descents. For any $d$ dividing $n+1$, we let $\brac{n}{k}_d$ count the number of these circular permutations fixed by adding $\frac{n+1}{d}$ modulo $n+1$ to each number in the word of the permutation. Note that $\brac{n}{k}_1 = \brac{n}{k}$.

\begin{ex}
Let $n=5$ and $k=d=2$. Then $\brac{5}{2}_2$ counts the number of circular permutations of $\{0, \dots, 5\}$ with $3$ descents that are fixed upon adding $\frac{6}{2}=3$ (modulo $6$) to each number. There are $6$ of these: $042315$, $015342$, $045312$, $021354$, $051324$, and $024351$.

Compare this to $2\Delta_{\frac32, 3}$, the intersection of the cube $[0,2]^3$ with the plane $x_1+x_2+x_3 = 3$, which is a hexagon of normalized volume $6$.
\end{ex}

\begin{rmk}
The definition given above is attributed to Postnikov in \cite{ESV} (cf. \cite{LamPostnikov}). One can also obtain these numbers using the cyclic sieving phenomenon on permutations with fixed cycle type and number of excedances described in \cite{SSW} as follows: if $w=w_1\cdots w_{n}$ is a permutation in $S_{n}$ with $k$ \emph{ascents}, then $(0 \; w_1 \; w_2 \; \cdots \; w_{n})$ is a permutation of $\{0, 1, \dots, n\}$ with cycle type $\lambda = (n+1)$ and $k+1$ \emph{excedances}. The cyclic action generated by conjugation by $(0 \; 1\; 2 \; \cdots \; n-1)$ as in \cite{SSW} then corresponds to the addition of a constant modulo $n+1$ to each letter of a circular permutation as described in the previous paragraph.
\end{rmk}

Analogous to Theorem~\ref{thm-stanley}, we have the following theorem.
\begin{thm} \label{thm-fractional}
The normalized volume of $d\Delta_{\frac{c}{d},n}$ is $\brac{dn-1}{c-1}_d$.
\end{thm}
\begin{proof}
The polytope $d\Delta_{\frac{c}{d},n}$ is defined by $0 \leq x_i \leq d$ for $1 \leq i \leq n$, and $\sum x_i = c$. We may slice this polytope by the hyperplanes $x_i=a$ for all $x_i$ and all integers $a$. Then each of the resulting regions is a translate of a hypersimplex: the region containing a general point $(x_1, \dots, x_n)$ is a translate of $\Delta_{k,n}$ by $(\lfloor x_1 \rfloor, \dots, \lfloor x_n \rfloor)$, where $k = c-\sum \lfloor x_i \rfloor$. Hence, using Theorem~\ref{thm-stanley}, the normalized volume of $d\Delta_{\frac{c}{d},n}$ equals
\[\sum_{k=1}^{n-1} \brac{n-1}{k-1} \cdot \#\{(x_1, \dots, x_n) \mid \textstyle\sum x_i = c-k, \quad x_i \in \{0, 1, \dots, d-1\}\}.\]
In other words, the normalized volume of $d\Delta_{\frac{c}{d},n}$ is the number of pairs $(w, x)$, where $w \in S_{n-1}$ has $k-1$ descents and $x=(x_1, \dots, x_n)$ is an integer point with $\sum x_i = c-k$ and $x_i \in [0, d-1]$. It therefore suffices to show that the number of such pairs $(w,x)$ is $\brac{dn-1}{c-1}_d$. We consider $w$ and $x$ to be periodic sequences with period $n$ (with $w_0=0$). 

Let $p=(p_0, p_1, \dots)$ be the sequence such that $p_0=0$, and $p_i - nx_i$ is the smallest integer congruent to $w_i$ modulo $n$ that is larger than $w_{i-1}$. For instance, for $n=6$, if $w = w_0w_1w_2\ldots = 014352\dots$ and $x = x_1x_2\ldots = 010021\dots$, then \[p = (0, 1, 10, 15, 17, 32, \quad 42, 43, 52, 57, 59, 74, \quad \dots).\]

Note that $p$ is strictly increasing and $p_{i}-p_{i-1} < n(x_i+1) \leq dn$ for all $i$. Moreover, the number of multiples of $n$ in the interval $(p_{i-1}, p_i]$ is $x_i$ if $w_{i-1}<w_i$, and $x_i+1$ if $w_{i-1}>w_i$. Since the permutation $w \in S_{n-1}$ has exactly $k-1$ descents (so as a circular permutation it has $k$ descents), we have that $p_n = (k+\sum x_i)n = cn$, so $p_{i+n} = p_i + cn$ for all $i$.

Then consider the sequence $\bar p$ obtained by reducing the terms in the sequence $p$ modulo $dn$. Since $c$ and $d$ are relatively prime and $p_0, \dots, p_{n-1}$ are all different modulo $n$, it follows that $\bar p_0, \dots, \bar p_{dn-1}$ are all distinct (as residues modulo $dn$). Moreover, since $p_{dn} = cdn$ and $p_i-p_{i-1} < dn$, $\bar p$ has exactly $c$ circular descents (corresponding to whenever there is a multiple of $dn$ in the interval $(p_{i-1}, p_i]$). Finally, since $p_{i+n} = p_i+cn$, adding $n$ (modulo $dn$) to each $\bar p_i$ performs a cyclic shift by $c'n$ to $\bar p$, where $c' \equiv c^{-1} \pmod d$. It follows that $\bar p$ is one of the circular permutations counted by $\brac{dn-1}{c-1}_d$.

Conversely, starting with a circular permutation $\bar p$ counted by $\brac{dn-1}{c-1}_d$, we can form the sequence $p$ by letting $p_i$ be the smallest integer larger than $p_{i-1}$ congruent to $\bar p_i$ modulo $dn$. Then since $\bar p$ has $c$ circular descents, $p_{dn} = cdn$. By the symmetry property of $\bar p$, we must have that $p_{i+n}-p_i = cn$. (Note that for $\bar p$ to be a permutation, we must have $c$ relatively prime to $d$.) It is now easy to reverse the procedure above to recover $(w, x)$, completing the proof of the bijection.
\end{proof}

Note that it follows from this proof that $\brac{n}{k}_d = 0$ if $d$ and $k+1$ are not relatively prime.

We can now mimic the proof of Theorem~\ref{thm-main} to get a result about $d$-sparse Laurent polynomials, as in \cite{ESV}.

\begin{thm}[\cite{ESV}]
The expected number of $d$-sparse Laurent polynomials $f(z)$ such that $\lbrac f^1\rbrac = \lbrac f^2\rbrac = \cdots = \lbrac f^{m+n-1}\rbrac=0$ is $\brac{m+n-1}{m-1}_d$.
\end{thm}
\begin{proof}
Let $N=m+n$. If $f(z)$ is $d$-sparse, then we can write
\[f(z) = z^{-m} \prod_{i=1}^{N/d} (1+r_iz^d)\]
for some nonzero $r_1, \dots, r_{N/d}$ with product 1, and again, as in Theorem~\ref{thm-main}, it suffices to show that the expected number of possible $r$ is $(\frac{N}{d})! \cdot \brac{N-1}{m-1}_d$.

 As in Theorem~\ref{thm-main}, we wish to find the coefficient of $z^{mk}$ in $\prod_{i=1}^{N/d} (1+r_iz^d)^k$ for $k = 1, \dots, N-1$ as a polynomial in the $r_i$. This polynomial will be nonzero when $d$ divides $mk$. Hence if $d$ is not relatively prime to $m$, this will occur for more than $\frac{N}{d}-1$ values of $k$, so we would not expect any such $f$ to exist.

If $d$ is relatively prime to $m$, then for $k=dk'$, let $P_{k'}(r_1, \dots, r_{\frac{N}{d}})$ be the coefficient of $z^{mk}$ in $\prod_{i=1}^{N/d}(1+r_iz^d)^k$, or equivalently the coefficient of $z^{mk'}$ in $\prod_{i=1}^{N/d}(1+r_iz)^{dk'}$. Then $(a_1, \dots, a_{\frac{N}{d}})$ lies in the support of $P_{k'}$ when $0 \leq a_i \leq dk'$ and $\sum_{i=1}^{N/d} a_i = mk'$. In other words, the Newton polytope of $P_{k'}$ is $k'\cdot d\Delta_{\frac{m}{d},\frac{N}{d}}$. Thus the expected number of such $f$ is
\[MV(d\Delta_{\frac{m}{d},\frac{N}{d}},\;\; 2\cdot d\Delta_{\frac{m}{d},\frac{N}{d}},\;\; \dots,\;\; (\tfrac{N}{d}-1) \cdot d\Delta_{\frac{m}{d},\frac{N}{d}},\;\; I),\]
where $I$ is the segment between $(0,0,\dots,0)$ and $(1,1,\dots,1)$. Then the same argument as in the last paragraph of the proof of Theorem~\ref{thm-main} shows that this mixed volume is $(\frac{N}{d})!$ times the normalized volume of $d\Delta_{\frac{m}{d},\frac{N}{d}}$, so the result follows from Theorem~\ref{thm-fractional}.
\end{proof}

\section{Concluding remarks}

Given the apparent connection between Laurent polynomials $f(z)$ that satisfy Theorem~\ref{thm-main} and Eulerian numbers, it is natural to ask whether there is any direct correspondence. In other words, can one give a bijection between permutations in $S_{n-1}$ with $k-1$ descents and these Laurent polynomials?

It is perhaps worth noting that the proof of Theorem~\ref{thm-stanley} given in \cite{Stanley} is bijective: it is shown that a hypersimplex can be triangulated in such a way that there is a natural bijection between the simplices in the triangulation and permutations with a given number of descents. (For more information about this triangulation, see \cite{LamPostnikov}.) Moreover, one can prove Theorem~\ref{thm-bernstein} using homotopy continuation methods, from which one may similarly be able to derive a bijection between the Laurent polynomials $f(z)$ and certain simplices. However, the details behind this line of reasoning have yet to be explored.

\section{Acknowledgments}
The author would like to thank Cynthia Vinzant and Daniel Erman for useful discussions. This work was partially supported by an NSF MSPRF award number DMS 1004375 and an NSF RTG grant number DMS 0943832.

\bibliography{eulerian}
\bibliographystyle{plain}

\end{document}